\documentclass{ejpamPRE}

\usepackage{mathrsfs}
\usepackage{euscript}

\begin{document}

\title{Generalized Tschebyscheff of the second kind and Bernstein polynomials change of bases}

\author[M. AlQudah]{Mohammad A. AlQudah\affil{1}\comma\corrauth}

\address{\affilnum{1}\ Department of Mathematics, Northwood University, Midland, MI 48640 USA
}

\emails{{\tt alqudahm@northwood.edu}\ (M. AlQudah), 
}

\begin{abstract}
We construct multiple representations relative to different bases of the generalized Tschebyscheff polynomials of second kind. Also, we provide an explicit closed from of The generalized Polynomials of degree $r$ less than or equal $n$ in terms of the Bernstein basis of fixed degree $n.$
In addition, we create the change-of-basis matrices between the generalized Tschebyscheff of the second kind polynomial basis and Bernstein polynomial basis. 
\end{abstract}

\keywords{Generalized Tschebyscheff, Bernstein Basis, Basis Transformation, B\'{e}zier Coefficient, Gamma function}
\ams{42C05, 33C50, 33C45, 33C70, 05A10, 33B15}

\maketitle

\section{Introduction, Background and Motivation}
It is possible to approximate a complicated continuous functions defined over finite domains by a polynomial and make the error less than a given accuracy. On the other side, polynomials can be characterized in many different bases such as the power product, Bernstein basis, and Tschebyscheff basis form, where every type of polynomial basis has its strength, advantages, and sometimes disadvantages.

It is useful to switch bases and work with more than one basis for a given polynomial; it is of vital importance in the efficiency of mathematical calculations, since many difficulties can be solved and many problems can be removed. 

\subsection{Bernstein Polynomials}
The $n+1$ polynomials $B_{k}^{n}(x)$ of degree $n,$ $x\in [0,1], k=0,1,\dots,n,$ defined as  
\begin{equation}\label{binomial} B_{k}^{n}(x)=\frac{n!}{k!(n-k)!} x^{k}(1-x)^{n-k}, \hspace{.2in} k=0,1,\dots,n ,\end{equation}
are called Bernstein polynomials.

There are a fair amount of literature on Bernstein polynomials, they are known for their geometric and analytical properties, see \cite{Farouki5} for more details. 
Analytic and geometric properties of Bernstein polynomials make them important for the development of B\'{e}zier curves and surfaces. The Bernstein polynomials are the standard basis for the B\'{e}zier representations of curves and surfaces in Computer Aided Geometric Design.  However, the Bernstein polynomials are not orthogonal and could not be used effectively in the least-squares approximation \cite{Rice}.
Since then the method of least squares approximation accompanied by orthogonal polynomials has been introduced and developed.

\subsection{Least-Square Approximation}
In the following definition, we define the continuous least-square approximations of a function $f(x)$ by using polynomials with standard power basis, $\{1,x,x^{2},\dots,x^{n}\}.$
\begin{definition}
For a function $f(x),$ continuous on $[0,1]$ the least square approximation requires finding a least-squares polynomial $p_{n}^{*}(x)=\sum_{k=0}^{n}a_{k}\phi_{k}(x)$
that minimizes the error $E(a_{0},a_{1},\dots,a_{n})=\int_{0}^{1}[f(x)-p_{n}^{*}(x)]^{2}dx,$
are called least squares approximations. 
\end{definition}
A necessary condition for $E(a_{0},a_{1},\dots,a_{n})$ to have a minimum over all values $a_{0},a_{1},\dots,a_{n},$ is $\frac{\partial E}{\partial a_{k}}=0.$
But,
\begin{equation*}
\frac{\partial E}{\partial a_{k}}=-2\int_{0}^{1}[f(x)-p_{n}^{*}(x)]\phi_{k}(x)dx, \hspace{.2in} k=0,\dots,n.
\end{equation*}
Thus, for $i=0,1,\dots,n,$ $a_{i}$ that minimize $\left\|f(x)-\sum_{k=0}^{n}a_{k}\phi_{k}(x)\right\|_{2}$ satisfy the system 
$$\int_{0}^{1}f(x)\phi_{i}(x)dx=\sum_{k=0}^{n}a_{k}\int_{0}^{1}\phi_{k}(x)\phi_{i}(x)dx.$$
which gives a system of $(n+1)$ equations, called normal equations, in $(n+1)$ unknowns: $a_{i},$ $i=0,\dots,n.$  Those $(n+1)$ unknowns of the least-squares polynomial $p_{n}^{*}(x),$  can be found by solving the normal equations.
By choosing $\phi_{i}(x)=x^{i},$ as a basis, then 
$$\int_{0}^{1}f(x)x^{i}dx=\sum_{k=0}^{n}a_{k}\int_{0}^{1}x^{i+k}dx=\sum_{k=0}^{n}\frac{a_{k}}{i+k+1}.$$
The coefficients matrix of the normal equations is Hilbert matrix which has round-off error difficulties and notoriously ill-conditioned for even modest values of $n.$
However, such computations can be made effective by using orthogonal polynomials. Thus,  choosing $\{\phi_{0}(x),\phi_{1}(x),\dots,\phi_{n}(x)\}$ to be orthogonal simplifies the least-squares approximation problem. The coefficients matrix of the normal equations will be diagonal, which gives a compact form for $a_{i},i=0,1,\dots,n.$ See \cite{Rice} for more details on the least squares approximations.
\subsection{Gamma Functions}
The gamma function $\Gamma(n)$ is an extension of the factorial function, with its argument shifted down by 1. That is, if $n$ is a positive integer: $\Gamma(n) = (n-1)!.$
The Eulerian integral of the first kind is useful and will be used in main result simplifications.
\begin{definition}
The Eulerian integral of the first kind is a function of two complex variables defined by
\begin{equation}\label{eulerian-int}\int_{0}^{1}u^{x-1}(1-u)^{y-1}du=\frac{\Gamma(x)\Gamma(y)}{\Gamma(x+y)},\hspace{.5in}\Re(x),\Re(y)>0.\end{equation}
\end{definition}

The double factorial of an integer $n$ is given by
\begin{equation}\label{douvle-factrrial}
\begin{aligned}
(2n-1)!!&=(2n-1)(2n-3)(2n-5)\dots(3)(1) \hspace{.2in}  \text{if $n$ is odd} \\
n!!&=(n)(n-2)(n-4)\dots(4)(2) \hspace{.63in} \text{if $n$ is even},
\end{aligned}
\end{equation}
where $0!!=(-1)!!=1.$ Using \eqref{douvle-factrrial}, we can derive the following relation
\begin{equation}\label{doublefac} n!!=\left\{\begin{array}{ll}2^{\frac{n}{2}}(\frac{n}{2})! & \text{if $n$ is even} \\\frac{n!}{2^{\frac{n-1}{2}}(\frac{n-1}{2})!} & \text{if $n$ is odd} \end{array}\right..\end{equation}
It is easy to derive the factorial of an integer plus half as 
\begin{equation}\label{n+half}
\left(n+\frac{1}{2}\right)!=\frac{\sqrt{\pi}}{2^{n+1}}(2n+1)!!.
\end{equation}
From the relation \eqref{doublefac} to have $(2n)!!=2^{n}n!,$
and $(2n)!=(2n-1)!!2^{n}n!.$
\subsection{Univariate Tschebyscheff-II and the Generalized Tschebyscheff-II Polynomials}
The univariate classical Tschebyscheff-II orthogonal polynomials $U_{n}(x)$ are special case of Jacobi polynomials $P_{n}^{(\alpha,\beta)}$ with $\alpha=\beta=1/2,$ where the inter-relationship between Tschebyscheff-II and Jacobi Polynomials given as
$P_{n}^{(\frac{1}{2},\frac{1}{2})}(1)U_{n}(x)=(n+1)P_{n}^{(\frac{1}{2},\frac{1}{2})}(x).$
Tschebyscheff-II polynomials are traditional defined on $[-1,1],$ however, it is more convenient to use $[0,1].$

For the convenience we recall the following explicit expressions for univariate Tschebyscheff-II polynomials of degree $n$ in $x$, using combinatorial  notation that gives more compact and readable formulas, see Szeg\"{o} \cite{Szego}:
\begin{equation}
U_{n}(x):=\frac{(n+1)(2n)!!}{(2n+1)!!}\sum_{k=0}^{n}\binom{n+\frac{1}{2}}{n-k}\binom{n+\frac{1}{2}}{k}\left(\frac{x+1}{2}\right)^{n-k}\left(\frac{x-1}{2}\right)^{k},
\end{equation}
which it can be transformed in terms of Bernstein basis on $x\in[0,1]$,
\begin{equation}\label{Tschebyscheff-II-basis fromat}
U_{n}(2x-1):=\frac{(n+1)(2n)!!}{(2n+1)!!}\sum_{k=0}^{n}(-1)^{n+1}\frac{\binom{n+\frac{1}{2}}{k}\binom{n+\frac{1}{2}}{n-k}}{\binom{n}{k}}B_{k}^{n}(x).
\end{equation}
The Tschebyscheff-II polynomials $U_{n}(x)$ of degree $n$ are the orthogonal polynomials, except for a constant factor, with respect to the weight function $\mathrm{W}(x)=\sqrt{1-x^{2}}.$
Also, the Tschebyscheff-II polynomials satisfy the orthogonality relation \cite{Gradshtein}
\begin{equation}\label{ortho-rel}
\int_{0}^{1}x^{\frac{1}{2}}(1-x)^{\frac{1}{2}}U_{n}(x)U_{m}(x)dx=\left\{\begin{array}{ll} 
0             & \mbox{if } m\neq n \\
\frac{\pi}{8} & \mbox{if } m=n
\end{array}\right..
\end{equation}

The generalized Tschebyscheff-II polynomials been characterization in \cite{Alqudahn:CharaII}, for $M,N\geq 0,$ the generalized Tschebyscheff-II polynomials 
$\left\{\mathscr{U}_{n}^{(M,N)}(x)\right\}_{n=0}^{\infty}$ are orthogonal on $[-1,1]$ with respect to the generalized weight function \cite{Koornwinder},
\begin{equation}\label{gen-wgt}
\frac{2}{\pi}(1-x)^{\frac{1}{2}}(1+x)^{\frac{1}{2}}+M\delta(x+1)+N\delta(x-1).
\end{equation}
and defined in \cite{Alqudahn:CharaII} as 
\begin{equation}\label{gen-Tschebyscheff-II-ccomb}
\mathscr{U}_{n}^{(M,N)}(x)=\frac{(2n+1)!!}{2^{n}(n+1)!}U_{n}(x)+\sum_{k=0}^{n} \lambda_{k} \frac{(2k+1)!!}{2^{k}(k+1)!}U_{k}(x),
\end{equation}
where
\begin{equation}\label{lmda}\lambda_{k}=\frac{k(k+1)(2k+1)(M+N)}{6}+\frac{(k+2)(k+1)^{2}k^{2}(k-1)MN}{9}.
\end{equation}
\section{Main Results}
In this section we provide a closed form for the matrix transformation of the generalized Tschebyscheff-II polynomial basis into Bernstein polynomial basis, and for Bernstein polynomial basis into generalized Tschebyscheff-II polynomial basis.
\subsection{Bernstein to Generalized Tschebyscheff-II Transformation and Vice Versa}
Rababah \cite{rababah4} provided some results concerning the univariate Tschebyscheff polynomials of first kind with respect to the weight function $(1-x^{2})^{1/2}.$ In this paper we extend the procedure in \cite{rababah4} to generalize the results for the generalized Tschebyscheff-II polynomials $\mathscr{U}_{r}^{(M,N)}(x)$ with respect to the generalized weight function \eqref{gen-wgt}. 

The next theorem, see \cite{Alqudahn:CharaII} for the proof, provides a closed form for generalized Tschebyscheff-II polynomial $\mathscr{U}_{r}^{(M,N)}(x)$ of degree $r$ as a linear combination of the Bernstein polynomials $B_{i}^{r}(x), i=0,1,\dots,r.$
\begin{theorem}\label{gen-jacinBer form}\cite{Alqudahn:CharaII}
For $M,N\geq 0,$ the generalized Tschebyscheff-II polynomials $\mathscr{U}_{r}^{(M,N)}(x)$ of degree $r$ have the following Bernstein representation:
\begin{equation}\label{gen-jac-inBer-r}
\mathscr{U}_{r}^{(M,N)}(x)=\frac{(2r+1)!!}{2^{r}(r+1)!}\sum\limits_{i=0}^{r}(-1)^{r-i}\vartheta_{i,r}B_{i}^{r}(x)+\sum_{k=0}^{r} \lambda_{k}\frac{(2k+1)!!}{2^{k}(k+1)!} \sum\limits_{i=0}^{k}(-1)^{k-i}\vartheta_{i,k}B_{i}^{k}(x)
\end{equation}
where $\lambda_{k}$ defined by \eqref{lmda}, $\vartheta_{0,r}=\frac{(2r+1)}{2^{2r}}\binom{2r}{r},$ and
$$\vartheta_{i,r}=\frac{(2r+1)^{2}}{2^{2r}(2r-2i+1)(2i+1)}\frac{\binom{2r}{r}\binom{2r}{2i}}{\binom{r}{i}}, \hspace{.05in}i=0,1,\dots,r.$$
The coefficients $\vartheta_{i,r}$ satisfy the recurrence relation
\begin{equation}\label{recur}\vartheta_{i,r}=\frac{(2r-2i+3)}{(2i+1)}\vartheta_{i-1,r}, \hspace{.1in}i=1,\dots,r.\end{equation}
\end{theorem}
Now, the next theorem used to combine the superior performance of the least-squares of the generalized Tschebyscheff-II polynomials with the geometric properties of the Bernstein polynomials basis.
\begin{theorem}
The entries $M_{i,r}^{n}, i,r=0,1,\dots,n$ of the matrix transformation of the generalized Tschebyscheff-II polynomial basis into Bernstein polynomial basis of degree $n$ are given by
\begin{equation}\label{gen-mu}
M_{i,r}^{n}=\Phi_{i,n}^{r}+\sum_{k=0}^{r}\lambda_{k}\Phi_{i,n}^{k},
\end{equation}
where $\lambda_{k}$ defined in \eqref{lmda} and
$$\Phi_{i,n}^{r}=\frac{(2r+1)!!}{2^{r}(r+1)!}\sum_{k=\max(0,i+r-n)}^{\min(i,r)}(-1)^{r-k}\frac{\binom{n-r}{i-k}\binom{r+\frac{1}{2}}{k}\binom{r+\frac{1}{2}}{r-k}}{\binom{n}{i}}.$$
\end{theorem}
\begin{proof}
A polynomial $p_{n}(x), x\in[0,1]$ of degree $n,$ can be written as as a linear combination of the Bernstein polynomial basis $p_{n}(x)=\sum_{r=0}^{n} c_{r}B_{r}^{n}(x)$ and the generalized Tschebyscheff-II polynomials $p_{n}(x)=\sum_{i=0}^{n}d_{i}\mathscr{U}_{i}^{(M,N)}(x).$

Need to find the matrix $M$ that maps the generalized Tschebyscheff-II coefficients $\{d_{i}\}_{i=0}^{n}$ into the Bernstein coefficients
$\{c_{r}\}_{r=0}^{n},$ 
\begin{equation}\label{gen-c_jac}
c_{i}=\sum_{r=0}^{n}M_{i,r}^{n}d_{r},
\end{equation}
which can be written in matrix format as
\begin{equation}\label{gen-c_jac-matrix}
\begin{bmatrix}
    c_{0}  \\
    c_{1}\\
    \vdots  \\
    c_{n} 
\end{bmatrix}=
\begin{bmatrix}
    M_{0,0}^{n} & M_{0,1}^{n} & M_{0,2}^{n} & \dots  & M_{0,n}^{n} \\
    M_{1,0}^{n} & M_{1,1}^{n} & M_{1,2}^{n} & \dots  & M_{1,n}^{n} \\
    \vdots & \vdots & \vdots & \ddots & \vdots \\
    M_{n,0}^{n} & M_{n,1}^{n} & M_{n,2}^{n} & \dots  & M_{n,n}^{n}
\end{bmatrix}.
\begin{bmatrix}
    d_{0}  \\
    d_{1}\\
    \vdots  \\
    d_{n} 
\end{bmatrix}.\end{equation}

But, the generalized Tschebyscheff-II polynomials \eqref{gen-Tschebyscheff-II-ccomb} can be written as a linear combination of the Bernstein polynomial basis as
\begin{equation}\label{gen-tche}\mathscr{U}_{r}^{(M,N)}(x)=\sum_{i=0}^{n}N_{r,i}^{n}B_{i}^{n}(x), \hspace{.1in} r=0,1,\dots,n,\end{equation} where the the $(n+1)\times(n+1)$ basis conversion matrix $N$ formed by the entries $N_{r,i}^{n}.$ Thus, the elements of $c$ can be written in the form 
\begin{equation}\label{gen-c_jac-2} c_{i}=\sum_{r=0}^{n}d_{r}N_{r,i}^{n}.
\end{equation}
Comparing \eqref{gen-c_jac} and \eqref{gen-c_jac-2}, we have $M_{i,r}^{n}=N_{r,i}^{n},$ for $i,r=0,\dots,n.$ 

Since each Bernstein polynomial of degree $r\leq n$ can be written in terms of Bernstein polynomials of degree $n$ using the following degree elevation defined by \cite{Farouki3}:
\begin{equation}\label{gen-ber-elv} B_{k}^{r}(x)=\sum_{i=k}^{n-r+k}\frac{\binom{r}{k}\binom{n-r}{i-k}}{\binom{n}{i}}B_{i}^{n}(x),\hspace{.1in}k=0,1,\dots,r.
\end{equation}
Substituting \eqref{gen-ber-elv} into \eqref{gen-jac-inBer-r}
and rearrange the order of summations, we find the entries
\begin{equation}\label{N-defn}\begin{aligned}
N_{r,i}^{n}&=\binom{n}{i}^{-1}\frac{(2r+1)!!}{2^{r}(r+1)!}\sum_{k=\max(0,i+r-n)}^{\min(i,r)}(-1)^{r-k}\binom{n-r}{i-k}\binom{r+\frac{1}{2}}{k}\binom{r+\frac{1}{2}}{r-k}\\
&+\sum_{k=0}^{r}\lambda_{k}\binom{n}{i}^{-1}\frac{(2k+1)!!}{2^{k}(k+1)!}\sum_{j=\max(0,i+k-n)}^{\min(i,k)}(-1)^{k-j}\binom{n-k}{i-j}\binom{k+\frac{1}{2}}{j}\binom{k+\frac{1}{2}}{k-j}
.\end{aligned}
\end{equation}
Therefore, the entries of the matrix $M$ are given by
$
M_{i,r}^{n}=\Phi_{i,n}^{r}+\sum_{k=0}^{r}\lambda_{k}\Phi_{i,n}^{k},
$
where
$$\Phi_{i,n}^{k}=\frac{(2k+1)!!}{2^{k}(k+1)!}\sum_{j=\max(0,i+k-n)}^{\min(i,k)}(-1)^{k-j}\frac{\binom{n-k}{i-j}\binom{k+\frac{1}{2}}{j}\binom{k+\frac{1}{2}}{k-j}}{\binom{n}{i}}.$$
\end{proof}
Now, we have the following corollary which enables us to write Tschebyscheff-II polynomials of degree $r \leq n$ in terms of Bernstein polynomials of degree $n.$
\begin{corollary}\label{gen-rabab2}
The generalized Tschebyscheff-II polynomials $\mathscr{U}_{0}^{(M,N)}(x),\dots,\mathscr{U}_{n}^{(M,N)}(x)$\ of degree less than or equal to $n$ can be expressed in the Bernstein basis of fixed degree $n$ by the following formula
$$\mathscr{U}_{r}^{(M,N)}(x)=\sum\limits_{i=0}^{n}N_{r,i}^{n}B_{i}^{n}(x),\hspace{.05in} r=0,1,\ldots,n$$
where 
\begin{equation*}
\begin{aligned}
N_{r,i}^{n}&=\frac{(2r+1)!!}{2^{r}(r+1)!}\sum_{k=\max(0,i+r-n)}^{\min(i,r)}
\frac{(-1)^{r-k}(2r+1)^{2}}{2^{2r}(2r-2k+1)(2k+1)}\frac{\binom{n-r}{i-k}\binom{2r}{r}\binom{2r}{2k}}{\binom{n}{i}}\\
&+\sum_{k=0}^{r}\lambda_{k}\frac{(2k+1)!!}{2^{k}(k+1)!}\sum_{j=\max(0,i+k-n)}^{\min(i,k)}
\frac{(-1)^{k-j}(2k+1)^{2}}{2^{2k}(2k-2j+1)(2j+1)}\frac{\binom{n-k}{i-j}\binom{2k}{k}\binom{2k}{2j}}{\binom{n}{i}}
.\end{aligned}
\end{equation*}
\end{corollary}
\begin{proof}
From \eqref{gen-tche} in the proof of the previous theorem, it is clear that each Tschebyscheff-II polynomial of degree $r\leq n$ can be written in terms of Bernstein polynomials of degree $n.$ 
Applying \eqref{n+half} with some simplifications, we have
\begin{equation*}
\binom{r+\frac{1}{2}}{k}\binom{r+\frac{1}{2}}{r-k}
=\frac{(2r+1)}{2^{r}(2r-2k+1)(r-k)!k!}\frac{(2r-1)!!}{(2k-1)!!}\frac{(2r+1)}{(2k+1)}\frac{(2r-1)!!}{(2(r-k)-1)!!}.
\end{equation*}
Using the fact $(2n)!=(2n-1)!!2^{n}n!$ we get 
$$\binom{r+\frac{1}{2}}{r-k}\binom{r+\frac{1}{2}}{k}=\frac{(2r+1)^{2}}{2^{2r}(2r-2k+1)(2k+1)}\binom{2r}{r}\binom{2r}{2k}.$$
Substituting the last identity into \eqref{N-defn} we get the desired result.
\end{proof}

The following theorem introduced in \cite{Alqudahn:CharaII} will be used to simplify a main result.
\begin{theorem}\label{gen-int-ber-jac}\cite{Alqudahn:CharaII}
Let $B_{r}^{n}(x)$ be the Bernstein polynomial of degree $n$ and $\mathscr{U}_{i}^{(M,N)}(x)$ be the generalized Tschebyscheff-II polynomial of degree $i,$ then
for $i,r=0,1,\dots,n$ we have
\begin{equation*}
\int_{0}^{1}x^{\frac{1}{2}}(1-x)^{\frac{1}{2}}B_{r}^{n}(x)\mathscr{U}_{i}^{(M,N)}(x)dx
=\Lambda_{r,n}^{i}+\sum_{d=0}^{i}\lambda_{d}\Lambda_{r,n}^{d},
\end{equation*}
where $\lambda_{d}$ defined in \eqref{lmda},
\begin{equation}\label{newlmda}
\Lambda_{r,n}^{d}=\binom{n}{r}
\frac{(2d+1)!!}{2^{d}(d+1)!}
\sum_{j=0}^{d}(-1)^{d-j}\binom{d+\frac{1}{2}}{j}\binom{d+\frac{1}{2}}{d-j}
\frac{\Gamma(r+j+\frac{3}{2})\Gamma(n+d-r-j+\frac{3}{2})}{\Gamma(n+d+3)},\end{equation}
and $\Gamma(x)$ is the Gamma function.
\end{theorem}

Finally, to write the Bernstein polynomial basis into generalized Tschebyscheff-II polynomial basis of degree $n,$ invert \eqref{gen-c_jac-matrix} and let $M_{i,r}^{n^{-1}},$ $N_{i,r}^{n^{-1}},$ $i,r=0,\dots,n$ be the entries of $M^{-1}$ and $N^{-1}$ respectively. The transformation of Bernstein polynomial into generalized Tschebyscheff-II polynomial basis of degree $n$ can then be written as
\begin{equation}\label{gen-Ber-Jac}B_{r}^{n}(x)=\sum_{i=0}^{n}N_{r,i}^{n^{-1}}\mathscr{U}_{i}^{(M,N)}(x).\end{equation}
To find the explicit closed form of $N_{r,i}^{n^{-1}},$ $i,r=0,1,\dots,n,$ multiply \eqref{gen-Ber-Jac} by $x^{\frac{1}{2}}(1-x)^{\frac{1}{2}}\mathscr{U}_{i}^{(M,N)}(x)$ and integrate over $[0,1]$ to have
\begin{equation}
\int_{0}^{1}x^{\frac{1}{2}}(1-x)^{\frac{1}{2}}B_{r}^{n}(x)\mathscr{U}_{i}^{(M,N)}(x)dx
=\sum_{i=0}^{n}N_{r,i}^{n^{-1}}\int_{0}^{1}x^{\frac{1}{2}}(1-x)^{\frac{1}{2}}\mathscr{U}_{i}^{(M,N)}(x)\mathscr{U}_{i}^{(M,N)}(x)dx.
\end{equation} 
Use the orthogonality relation \eqref{ortho-rel} to obatin
\begin{equation}
\int_{0}^{1}B_{r}^{n}(x)(1-x)^{\frac{1}{2}}x^{\frac{1}{2}}\mathscr{U}_{i}^{(M,N)}(x)dx
=\frac{\pi}{8}\left(\frac{(2i+1)!!}{2^{i}(i+1)!}\right)^{2}N_{r,i}^{n^{-1}}(1+\lambda_{i})^{2}.
\end{equation} 
Using \eqref{eulerian-int}, Theorem \ref{gen-int-ber-jac}, the fact that $M_{i,r}^{n}=N_{r,i}^{n},$ and $\Lambda_{r,n}^{d}$ defined in \eqref{newlmda} we get
\begin{equation}
M_{i,r}^{n^{-1}}=\frac{8}{\pi(1+\lambda_{i})^{2}}\left(\frac{2^{i}(i+1)!}{(2i+1)!!}\right)^{2}
\left(\Lambda_{r,n}^{i}
+\sum_{d=0}^{i}\lambda_{d}
\Lambda_{r,n}^{d}\right).
\end{equation}
Hence, we have the following theorem.
\begin{theorem}
The entries of the matrix of transformation of the Bernstein polynomial basis into the generalized Tschebyscheff-II polynomial basis
of degree $n$ are given by
\begin{equation*}
M_{i,r}^{n^{-1}}=\frac{8}{\pi(1+\lambda_{i})^{2}}\left(\frac{2^{i}(i+1)!}{(2i+1)!!}\right)^{2}
\left(\Lambda_{r,n}^{i}+\sum_{d=0}^{i}\lambda_{d}
\Lambda_{r,n}^{d}\right), \hspace{.1in} i,r=0,1,\dots,n.
\end{equation*}
\end{theorem}

\paragraph{{\bf}ACKNOWLEDGEMENTS} The author thanks the anonymous referees for their fruitful suggestions,
which immensely helped to improve the presentation of the paper.

\end{document}